\documentclass[12pt,a4paper,oneside]{article}
\usepackage[letterpaper,hmargin=1.00in,vmargin=1.00in]{geometry}
\usepackage[utf8]{inputenc}
\usepackage[T1]{fontenc}
\usepackage{amsmath, mathtools}
\usepackage{amsfonts}
\usepackage{amsthm}
\usepackage{amssymb, textcomp}
\usepackage{graphicx}
\usepackage{epstopdf}
\usepackage{enumerate}
\usepackage{float}
\usepackage{caption}
\usepackage{subcaption}
\usepackage{csquotes}
\usepackage{graphicx,epstopdf}
\usepackage{breqn}
\usepackage{multirow,array}

\newcolumntype{L}{>{\centering\arraybackslash}m{3cm}}
\theoremstyle{definition}

\newtheorem{theorem}{Theorem}[section]
\newtheorem{lemma}{Lemma}[section]
\newtheorem{corollary}{Corollary}[section]
\newtheorem{proposition}{Proposition}[section]

\begin{document}
 \begin{center}
		{\large On the $\mathcal{ABS}$ spectrum and energy of  graphs  }\\
		\vspace*{0.3cm} 
	\end{center}
	\begin{center}
		Swathi Shetty$^{1}$, B. R. Rakshith$^{*,2}$, Sayinath Udupa N. V. $^{3}$\\
		Department of Mathematics, Manipal Institute of Technology\\ Manipal Academy of Higher Education\\ Manipal, India -- 576104\\
		swathi.dscmpl2022@learner.manipal.edu$^{1}$\\
		ranmsc08@yahoo.co.in; rakshith.br@manipal.edu$^{*,2}$\\
		sayinath.udupa@manipal.edu	$^{3}$.
	\end{center}
 \begin{abstract}
Let $\eta_{1}\ge \eta_{2}\ge\cdots\ge \eta_{n}$ be the eigenavalues of $\mathcal{ABS}$ matrix. In this paper, we characterize connected graphs with $\mathcal{ABS}$ eigenvalue $\eta_{n}>-1$. As a result, we determine all connected graphs with exactly two distinct $\mathcal{ABS}$ eigenvalues. We show that a connected bipartite graph has three distinct $\mathcal{ABS}$ eigenvalues if and only if it is a complete bipartite graph. Furthermore, we present some bounds for the $\mathcal{ABS}$ spectral radius (resp. $\mathcal{ABS}$ energy) and characterize extremal graphs. Also, we obtain a relation between $\mathcal{ABC}$ energy and $\mathcal{ABS}$ energy. Finally, the chemical importance of $\mathcal{ABS}$ energy is investigated and it shown that the $\mathcal{ABS}$ energy is useful in predicting certain properties of molecules. 
  \bigskip
  
  
\end{abstract}

  \section{Introduction}\label{sec1}
  Throughout this article, we assume that $G$ is a graph with vertex set $V (G)$ and edge set $E(G)$, where $V(G)=\{v_{1},v_{2},\ldots,v_{n}\}$ and $|E(G)|=m$.   If two vertices $v_i$ and $v_j$ are adjacent, then
we write it as $v_i\sim v_j$ otherwise, $v_i\not\sim v_j$. We denote the degree of the vertex $v_{i}$ by $d(v_{i})/d_{i}$.\\ Adjacency matrix of $G$ is one of the well-studied graph matrix, denoted by $A(G)$ and defined as $A(G)=[a_{ij}]_{n\times n}$, where $a_{ij}=1$ if and only if $v_{i}\sim v_{j}$ or 0, otherwise.   
 If $\lambda_{1}\ge\lambda_{2}\ge\cdots\ge\lambda_{n}$ are the eigenvalues of $A(G)$, then the sum  $\displaystyle\sum\limits_{i=1}^n |\lambda_i|$ is called the energy of graph $G$ and is denoted by $\mathcal{E}(G)$. The concept of graph energy, introduced by Gutman in 1978, slowly attracted mathematicians and chemists. In recent years, extensive research on graph energy has been carried out. For recent research on graph energy, see \cite{estrada2017meaning, alawiah2018new, oboudi2019new, espinalgraph, akbariarelations, tavakoli2024energy} and refer to the book "Graph Energy" by Li, Shi, and Gutman \cite{li2012graph}. The study of graph energy is extended to various graph matrices, including (signless) Laplacian matrix, distance matrix, degree-based graph matrices and distance-based graph matrices. More than 50 graph energies have been defined so far. See \cite{gutman2020research} for more details.\\[2mm]  
A topological index is a numerical quantity derived from the graph's structure. In literature, plenty of topological indices are defined and used as molecular descriptors (see \cite{gutman2013degree,das2023neighborhood,gutman2021geometric,ali2020symmetric} ).  Most of the degree-based topological indices
 can be represented as $TI(G)=\sum_{v_{i}\sim v_{j}}\mathcal{F}(d_{i},d_{j})$, where $\mathcal{F}(d_{i},d_{j})= \mathcal{F}(d_{j},d_{i})$. As examples, we have first Zagreb index $\mathcal{F}(d_{i},d_{j})=d_{i}+d_{j}$, second Zagreb index $\mathcal{F}(d_{i},d_{j})=d_{i}d_{j}$, Randi\'c index ($R(G)$) $\mathcal{F}(d_{i},d_{j})=\dfrac{1}{\sqrt{d_{i}d_{j}}}$, harmonic index ($H(G)$) $\mathcal{F}(d_{i},d_{j})=\dfrac{2}{d_{i}+d_{j}}$, sum-connectivity index ($\chi(G)$) $\mathcal{F}(d_{i},d_{j})=\dfrac{1}{\sqrt{d_{i}+d_{j}}}$, atom-bound connectivity index $(\mathcal{ABC}(G))$ $\mathcal{F}(d_{i},d_{j})=\sqrt{\dfrac{d_{i}+d_{j}-2}{d_{i}d_{j}}}$, atom-bound sum-connectivity index $(\mathcal{ABS}(G))$ $\mathcal{F}(d_{i},d_{j})=\sqrt{\dfrac{d_{i}+d_{j}-2}{d_{i}+d_{j}}}$, etc.
 For a topological index $TI(G)$, Das et al. \cite{das2018degree} defined a general extended adjacency matrix as $\mathcal{T}=(t_{ij})_{n\times n}$, where $t_{ij}=\mathcal{F}(d_{i},d_{j})$ if $v_{i}\sim v_{j}$ or $0$, otherwise. The sum of  absolute values of all the eigenvalues of the matrix $\mathcal{T}$ is called the energy of the general extended adjacency matrix $\mathcal{T}$. In \cite{das2018degree}, Das et al. obtained several lower and upper bounds for the energy of the matrix $\mathcal{T}$, and deduced several known results about degree-based energies of graphs.\\[2mm]
 The $\mathcal{ABS}$ index  was introduced recently by Ali et al. in \cite{ali2022atom}. It combines both sum-connectivity index  and atom-bound sum connectivity index. Bounds on $\mathcal{ABS}$ index for the classes of (molecular) trees and general graphs are obtained in \cite{ali2022atom} and also extremal graphs are classified.
 Chemical applicability of $\mathcal{ABS}$-index is demonstrated in \cite{ali2023atom,nithya2023smallest}. For more details about $\mathcal{ABS}$ index  we refer to the survey article \cite{ali2024extremal} by Ali et al. The $\mathcal{ABS}$ matrix of $G$ is defined to be the matrix $\mathcal{AS}(G)=(w_{ij})_{n\times n}$, where $w_{ij}=\sqrt{\dfrac{d_{i}+d_{j}-2}{d_{i}+d_{j}}}$ if $v_{i}\sim v_{j}$ and 0, otherwise. We denote the eigenvalues of $\mathcal{AS}(G)$ by $\eta_{1}\ge\eta_{2}\ge\cdots\ge\eta_{n}$. The sum $\sum_{i=1}^{n}|\eta_{i}|$ is called the $\mathcal{ABS}$ energy of $G$ and is denoted by $\mathcal{E}_{ABS}(G)$. The study of properties of $\mathcal{ABS}$ matrix began recently. In \cite{lin2024abs}, it is proved that the $\mathcal{ABS}$ Estrada index ($\sum_{i=1}^{n}e^{\eta_{i}}$) of trees is maximum from the star graph and it is minimum for the path graph. Also, in \cite{lin2024abstrees}, the authors proved that $\mathcal{ABS}$ spectral radius of a tree is maximum for star graph and it is minimum for the path graph. The chemical importance of the $\mathcal{ABS}$ Estrada
 index and the $\mathcal{ABS}$ spectral radius  are investigated separately in \cite{lin2024abs,lin2024abstrees}, and it is shown that the $\mathcal{ABS}$ Estrada index and  $\mathcal{ABS}$ spectral radius can be useful in  predicting  certain properties of molecules.\\[2mm]
 Motivated by this, in Section 2 of the paper, we characterize connected graphs with $\mathcal{ABS}$ eigenvalue $\eta_{n}>-1$. As a result, we determine all connected graphs with exactly two distinct $\mathcal{ABS}$ eigenvalues. Further, we show that a connected bipartite graph has three distinct $\mathcal{ABS}$ eigenvalues if and only if it is a complete bipartite graph. In Sections 3 and 4, we present some bounds for the $\mathcal{ABS}$ spectral radius (resp. $\mathcal{ABS}$ energy) and characterize extremal graphs. Also, we obtain a relation between $\mathcal{ABC}$ energy and $\mathcal{ABS}$ energy. In Section 5, the chemical importance of $\mathcal{ABS}$ energy is investigated and it shown that the $\mathcal{ABS}$ energy is useful to predict the boiling point and pi-electron energy of benzenoid hydrocarbons.\\[2mm]
 As usual, the complete graph, path graph and  complete bipartite graph  on  $n$ vertices are denoted by  $K_{n}$, $P_{n}$ and $K_{n_{1},n_{2}}$  ($n_{1}+n_{2}=n$), respectively.
\section{Properties of ABS eigenvalues}\label{sec3}
The following proposition is one of the basic properties of $\mathcal{ABS}$ eigenvalues. We omit its proof as it is straightforward.
 \begin{proposition}\label{ABS prop1}
	Let $G$ be a graph on n vertices. Let $\eta_1\ge \eta_2\ge\dots\ge \eta_{n}$ be its $\mathcal{ABS}$- eigenvalues. Then
	 $\sum\limits_{i=1}^n\eta_i=0$, $\sum\limits_{i=1}^n \eta_i^2=2(m-H(G))$ and $\sum\limits_{1\le i<j\le n}\eta_i\eta_j=H(G)-m$.\end{proposition}
 Let $M$ be a Hermitian matrix of order $n$. We denote the eigenvalues of $M$ by $\theta_{1}(M)\ge \theta_{2}(M)\ge\cdots\ge \theta_{n}(M)$. The following lemma is the well-known Cauchy's interlacing theorem.
 \begin{lemma}{\rm\cite{horn2012matrix}}\label{cit}
 	Let $M$ be a symmetric matrix of order $n$ and let $M_k$ be its leading principal $k \times k$
 	submatrix. Then $\theta_{n-k+i}(M)\le \theta_i(M_k)\le \theta_i(M) $ for $i=1,2,\dots,k$.
 \end{lemma}
 
 \begin{theorem}\label{one eig}
	Let $G$ be a graph on n vertices. Then the $\mathcal{ABS}$ eigenvalues of G are all equal if and only if  $G\cong p K_2\cup qK_1$, where $p+q=n$.
\end{theorem}
\begin{proof}
	Suppose that the $\mathcal{ABS}$ eigenvalues of $G$ are all equal. Then by Proposition~\ref{ABS prop1}, $\sum\limits_{i}^n \eta_i=0$, and so the   $\mathcal{ABS}$ eigenvalues of $G$ are  zeros. Let $H$ be a component of $G$. If $|V(H)|\ge 3$, then there exists a vertex $x$ in $H$ of degree at least two. Let $y$ be a vertex of $H$ adjacent to $x$. Then the principal minor of $\mathcal{AS}(G)$ corresponding to the vertices $x$ and $y$ is non-zero. Thus, by Cauchy's interlacing theorem (see Lemma \ref{cit}), the least eigenvalue of $\mathcal{AS}(G)$ is non-zero, a contradiction. Hence, $|V(H)|\le 2.$ Therefore, $G\cong pK_2\cup qK_1$, where $p+q=n$. 
	Conversely, if $G\cong pK_2\cup qK_1$, then all the  entries of $\mathcal{AS}(G)$ are zeros. Thus, $\eta_1=\eta_2=\dots=\eta_n=0.$
\end{proof}
The diameter of a graph $G$ is the maximum distance between any pair of vertices in $G$ and it is denoted by $diam(G)$. In the following theorem, we characterize connected graphs with $\eta_{n}(G)> -1$.
\begin{theorem}\label{minimum}
Let G be connected graph on n vertices. Then $\eta_{n}(G)> -1$ if and only if $G\cong K_{n}$ or $P_{3}$.	
\end{theorem}
\begin{proof}
Assume that $diam(G)\ge 2$ and  $G\ncong P_3$. Let $x-y-z$ be an induced path in $G$.  
Then either $d(y)\ge 3$, $d(x)\ge 2$ or $d(z)\ge 2$. Let $\mathcal{AS}[p,q,r]$ denote the principal submatrix of $\mathcal{AS}(G)$ corresponding to the vertices $p,q$ and $r$, where $p-q-r$ is an induced path in $G$. Let $\theta_{1}[p-q-r]\ge\theta_{2}[p-q-r]\ge\theta_{3}[p-q-r]$ be the  eigenvalues of $\mathcal{AS}[p,q,r]$. Then 
\begin{eqnarray*}
\theta_{1}[p-q-r]&=&\sqrt{2}\,\,\sqrt{1-\dfrac{1}{d(q)+d(r)}-\dfrac{1}{d(q)+d(p)}};\\
\theta_{2}[p-q-r]&=&0;\\
\theta_{3}[p-q-r]&=&-\sqrt{2}\,\,\sqrt{1-\dfrac{1}{d(q)+d(r)}-\dfrac{1}{d(q)+d(p)}}.
\end{eqnarray*}
Also, by Cauchy's interlacing theorem (see Lemma \ref{cit}), $\eta_{n}(G)\le\theta_{3}[p-q-r]$.
\begin{figure}[H]
\includegraphics[width=0.9\textwidth]{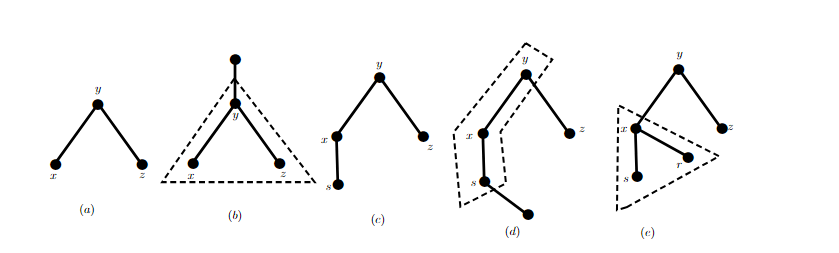}
\caption{Graphs considered in the proof of Theorem \ref{minimum}.}
\label{fig 1}
\end{figure}
\noindent
Case 1: Let $d(y)\ge 3$, $d(x)\ge1$ and $d(z)\ge1$ (see Fig.~\ref{fig 1}(b)). Then\\ $2\left(1-\dfrac{1}{d(y)+d(z)}-\dfrac{1}{d(y)+d(x)}\right)\ge 1$. Thus, 
$\eta_{n}(G)\le\theta_{3}[x-y-z]\le -1$.\\[4mm]
\noindent
Case 2: Let $d(y)=2$, $d(x)\ge2$ and $d(z)\ge1$. If $d(z)\ge 2$, then\\ $2\left(1-\dfrac{1}{d(y)+d(z)}-\dfrac{1}{d(y)+d(x)}\right)\ge 1$. So, $\eta_{n}(G)\le\theta_{3}[x-y-z]\le -1$. Otherwise, $d(z)=1$. Let $s$ be a vertex adjacent with the vertex $x$ in $G$ (see Fig. 1(c)). If $G\cong P_{4}$, then $\eta_{n}(G)=-1.0306<-1$. Suppose $G\ncong P_{4}$, then either $d(x)\ge3$ or $d(s)\ge 2$.\\[2mm]
Subcase 2.1: Let $d(x)\ge2$ and $d(s)\ge 2$ (see Fig. \ref{fig 1}(d)). Then\\ $2\left(1-\dfrac{1}{d(s)+d(x)}-\dfrac{1}{d(y)+d(x)}\right)\ge 1$. Thus, $\eta_n(G)\le\theta_3[s-x-y]\le -1$.\\[4mm]
Subcase 2.2: Let $d(x)\ge 3$ and $d(s)=1$. Then there exists a vertex $r$ adjacent with the vertex $x$ in $G$ (see Fig. 1(e)).Therefore, $2\left(1-\dfrac{1}{d(s)+d(x)}-\dfrac{1}{d(r)+d(x)}\right)\ge 1$. Thus, $\eta_n(G)\le\theta_3[s-x-r]\le -1$.\\[2mm]
Thus for a connected graph $G\ncong P_3$ with $diam(G)\ge 2$, $\eta_{n}(G)\le -1.$ Hence $G\cong K_{n}$ or $P_{3}$.  Conversely, $\eta_{n}(K_{n})=-\sqrt{\dfrac{n-2}{n-1}}>-1$ and $\eta_{3}(P_{3})=-0.81649>-1$. This completes the proof of the theorem. 
\end{proof}
\begin{corollary}\label{complete}
Let $G$ be a connected graph of order $n\ge2$. Then $\eta_n=-\sqrt{\dfrac{n-2}{n-1}}$ if and only if $G\cong K_n$.
\end{corollary}
\begin{proof}
Suppose $\eta_n=-\sqrt{\dfrac{n-2}{n-1}}$. Then by Theorem \ref{minimum}, $G\cong K_{n}$ or $P_{3}$. Since $n_{3}(P_{3})=-\sqrt{\dfrac{2}{3}}~\left(\neq -\sqrt{\dfrac{1}{2}}\right)$, $G\cong K_{n}$. The converse part is direct. 
\end{proof}
Let $B_1$ and $B_2$ be two real matrices of same order, we write $B_1\preceq B_2$ if every entry in $B_1$ does not exceed the counterpart in $B_2$. The following lemma is useful to prove our next result. \begin{lemma}{\rm\cite{horn2012matrix}}\label{b_1b_2}
	{Let $B_1, B_2$ be non-negative matrices of order $n$. If $B_1\preceq B_2$, then $\rho(B_1)\le \rho(B_2)$. Further, if $B_1$ is irreducible and $B_1\neq B_2$, then $\rho(B_1)< \rho(B_2)$.}
\end{lemma}

\begin{theorem}\label{complete 1}
Let $G$ be a connected graph of order $n>2$. Then $\mathcal{AS}(G)$ has exactly two distinct eigenvalues if and only if $G\cong K_n$.
\end{theorem}
\begin{proof}
Suppose $G$ has exactly two distinct eigenvalues. Since $G$ is a connected graph of order $n> 2$, the matrix $\mathcal{AS}(G)$ is irreducible, and thus by Perron-Frobenius theory its largest eigenvalue, i.e., $\eta_{1}(G)$ is a simple eigenvalue of $G$. So, $\eta_{1}(G)$ and $\eta_{n}(G)$ are the two distinct eigenvalues of $G$, and $\eta_{2}(G)=\eta_{3}(G)=\cdots=\eta_{n}(G)$.  
Let $B=\sqrt{\dfrac{n-2}{n-1}} A(K_n)$. Then the eigenvalues of $B$ are $(n-1)\sqrt{\dfrac{n-2}{n-1}},\underbrace{-\sqrt{\dfrac{n-2}{n-1}},\ldots,-\sqrt{\dfrac{n-2}{n-1}}}_{n-1}$. Since $\mathcal{AS}(G)\preceq B$, $\eta_{1}(G)\le (n-1)\sqrt{\dfrac{n-2}{n-1}}$ by Lemma \ref{b_1b_2}. Therefore, $-(n-1)\eta_{n}(G)\le(n-1)\sqrt{\dfrac{n-2}{n-1}}$. That is, $\eta_{n}(G)\ge -\sqrt{\dfrac{n-2}{n-1}}>-1$. Hence by Theorem \ref{minimum}, $G\cong K_{n}$ or $P_{3}$. So, $G\cong K_{n}$ because $P_{3}$ has three distinct eigenvalues. The converse part is straightforward.    
\end{proof}
The following lemma is important to prove our next result.
\begin{lemma}{\rm\cite{horn2012matrix}}\label{origin}
	Let $C\in M_{n,m}$, $q=min\left\{n,m\right\}$,  $\sigma_1\ge \sigma_2 \ge \dots \ge \sigma_q$ be the ordered singular values of $C$, and define the Hermitian matrix $\mathcal{H}=\begin{bmatrix}
		0 & C\\ C^* & 0
	\end{bmatrix}$. The ordered eigenvalues of $\mathcal{H}$ are $-\sigma_ 1\le \dots \le -\sigma_q\le 0=\dots= 0\le \sigma_q\le \dots \le\sigma_1$.
\end{lemma}
\begin{theorem}
Let $G$ be a graph of order $n$. Let $M=(m_{ij})_{n\times n}$ be a non-negative symmetric matrix of order $n$, where $m_{ij}$ is positive if and only if $v_{i} \sim v_{j}$. Then the graph $G$ is bipartite if and only if the eigenvalues of the matrix $M$ are symmetric about origin.
\end{theorem}
\begin{proof}Suppose $G$ is a  bipartite graph. Then the matrix $M$ can be written as $M=\begin{bmatrix} 0 & C\\
C^T & 0
\end{bmatrix}$, where $C$ is a rectangular matrix with non-negative entries. Therefore, by Lemma~\ref{origin}, the eigenvalues of $M$ are symmetric about the origin.\\
Conversely, suppose the eigenvalues of $M$ are symmetric about the origin. Then $trace(M^k)=0$ for all odd integer $k>0$.  By the definition of the matrix $M$, one can easily check that $trace(M^k(G))>0$ if and only if $trace(A^{k}(G))>0$. Now, assume that $G$ contains an odd cycle of length $k$. Then $trace(A^k(G))>0$ {(see \cite[Proposition 1.3.1]{brouwer2011spectra})}. So, $trace(M^{k})>0$, a contradiction. Thus, $G$ must be a bipartite graph.  
\end{proof}
The following corollary is immediate from the above theorem
\begin{corollary}\label{bipartite}
A graph $G$ is bipartite if and only if the eigenvalues of $\mathcal{AS}(G)$ are symmetric about origin.
\end{corollary}
\begin{theorem}\label{cbipartite}
A connected bipartite graph $G$ of order $n>2$ has three distinct $\mathcal{ABS}$ eigenvalues if and only if $G$ is a complete bipartite graph.
\end{theorem}
\begin{proof}
From Perron Frobenius theorem and by Corollary~\ref{bipartite}, $\eta_1(G)>0$ and $\eta_{n}(G)$ are simple $\mathcal{ABS}$ eigenvalues of $G$. Suppose $G$ has a non-zero $\mathcal{ABS}$ eigenvalue other than $\eta_{1}(G)$ and $\eta_{n}(G)$. Then by Corollary \ref{bipartite}, $G$ must have at least four distinct $\mathcal{ABS}$ eigenvalues. Thus, $0$ is an $\mathcal{ABS}$ eigenvalue of $G$ with multiplicity $n-2$.  Hence, $rank(\mathcal{AS}(G))=2$. Let $U$ and $W$ be the vertex partition sets of $G$. Let $u\in U$ and $w\in W$. Since $G$ is a connected bipartite graph, the rows corresponding to the vertices $u$ and $w$ are linearly independent. Further, since $rank(\mathcal{A(S)}(G))=2$ and $G$ is bipartite, the rows corresponding to the vertices belonging to $U$(respectively, $W$) are in the linear span of the row vector corresponding to the vertex $u$ (respectively, $w$). Thus, the vertices in $U$ (respectively, $W$) share the same vertex neighborhood set.
Assume that $w\in W$ and $w\notin N(u)$. Then $w$ is not adjacent with any vertices of $U$. Therefore, $w$ is an isolated vertex of $G$, a contradiction because $G$ is a connected graph of order at least $3$. Thus, $N(u)=W$ and $N(w)=U$. Therefore, $G$ is a complete bipartite graph.\\[2mm]
Conversely, if $G$ is the complete bipartite graph $K_{n_1,n_2}$ of order $n=n_{1}+n_{2} (\ge 3)$, then $\mathcal{AS}(G)=\sqrt{1-\dfrac{2}{n_1+n_2}}A(G)$. Therefore the $\mathcal{ABS}$ eigenvalues of $G$ are\\ $\sqrt{n_{1}n_{2}\left(1-\dfrac{2}{n_1+n_2}\right)},\underbrace{0,0,\ldots,0}_{n-2}, -\sqrt{n_{1}n_{2}\left(1-\dfrac{2}{n_1+n_2}\right)}$. Thus, $G$ has exactly three distinct $\mathcal{ABS}$ eigenvalues.
\end{proof}
\section{Bounds for the ABS spectral radius}\label{sec4}
 In this section, we give some bounds for the largest $\mathcal{ABS}$ eigenvalue $\eta_{1}(G)$.
 \begin{lemma}\label{eeeelemma}
Let G be graph of order $n$ with maximum degree $\Delta$ and minimum degree $\delta$. Then the row sums of $\mathcal{AS}(G)$ are equal if and only if $G$ is a regular graph.
\end{lemma}
\begin{proof}
Suppose the row sums of $\mathcal{AS}(G)$ are equal . Let $u,v\in V(G)$ such that $d(u)=\delta$ and $d(v)=\Delta$. Then\begin{eqnarray}\label{2eq1}
 	\sum\limits_{v_{i}:u\sim v_{i}}\sqrt{1-\dfrac{2}{\delta+d_i}}=\sum\limits_{v_{j}:v\sim v_{j}}\sqrt{1-\dfrac{2}{\Delta+d_j}}.
 \end{eqnarray}If $\delta\neq \Delta$, then $\sum\limits_{v_{j}:v\sim v_{j}}\sqrt{1-\dfrac{2}{\Delta+d_j}}\ge\Delta\sqrt{1-\dfrac{2}{\Delta+\delta}} >\sum\limits_{v_{i}:u\sim v_{i}}\sqrt{1-\dfrac{2}{\delta+d_i}}$, a contradiction to the equation (\ref{2eq1}). Thus, $\delta=\Delta$. i.e., $G$ is a regular graph. The converse part is straightforward.
 \end{proof}
 The following theorem gives a lower bound for $\eta_1(G)$ in terms of order and the atom-bond sum connectivity index of graph $G$.
\begin{theorem}\label{lbe2}
Let $G$ be a graph of order $n$, minimum degree $\delta$ and maximum degree  $\Delta$. Then 
   $\eta_1(G) \ge \dfrac{2\mathcal{ABS}(G)}{n}$. Further, equality holds if and only if $G$ is a regular graph.
\end{theorem}
\begin{proof}
    Let $x=(x_1,x_2,\dots,x_n)^T$ be a vector in $\mathbb{R}^{n}$. Then\begin{center}
        $x^T\mathcal{AS}(G)x=2\sum\limits_{v_i\sim v_j}\sqrt{\dfrac{d_i+d_j-2}{d_i+d_j}}x_ix_j$.
    \end{center}  
Set $x=(1,1,\dots,1)^T$.  Then by Rayleigh's inequality, $\eta_1(G)\ge \dfrac{x^T\mathcal{AS}(G) x}{x^Tx}=\dfrac{2\mathcal{ABS}(G)}{n}$, and the equality holds if and only if $x=(1,1,\dots,1)^{T}$ is an eigenvector of $\mathcal{AS}(G)$ corresponding to the eigenvalue $\eta_1(G)$. Suppose $\eta_{1}(G)=\dfrac{2\mathcal{ABS}(G)}{n}$. Then the row sums of  $\mathcal{AS}(G)$ are equal. Therefore, by Lemma \ref{eeeelemma}, $G$ is a regular graph.
\end{proof}
Next, we provide a lower and upper bound for $\eta_1(G)$ in terms of order, size and the harmonic index of graph $G$.
\begin{theorem}\label{lbe1}
	Let $G$ be a graph of order n  and size m with no isolated vertices. Then \begin{equation}\label{e1}
	\sqrt{\frac{2\left(m-H(G)\right)}{n}}\le 	\eta_1(G)\le \sqrt{\frac{2(n-1)}{n}\left(m-H(G)\right)}
	\end{equation}
	with equality holds if and only if $n$ is even and $G\cong \dfrac{n}{2}K_2$.   
\end{theorem}
\begin{proof}
	By Proposition~\ref{ABS prop1}, \begin{equation}\label{e2}
		n\eta_1^{2}\ge \sum\limits_{i=1}^n \eta_i^2=2(m-H(G)).
	\end{equation}
 Since $\sum\limits_{i=1}^n \eta_i=0$, $\eta_1^2=\left(\sum\limits_{i=2}^n \eta_i\right)^2 $. Therefore by Cauchy-Schwarz inequality,
   $\eta_1^2\le (n-1)\sum\limits_{i=2}^n \eta_i^2$ and the equality holds if and only if $\eta_{2}=\eta_{3}=\cdots=\eta_{n}$.
 So,\begin{align}\label{en3} n\eta_1^2 &\le 2(n-1)\left(m-H(G)\right).\end{align}
 Thus from equations (\ref{e2}) and (\ref{en3}) we get the desired inequality.  
	Suppose $n$ is even and $G\cong \dfrac{n}{2}K_2$. Then $\mathcal{AS}(G)$ is the null matrix, and so $\eta_1=\eta_2=\dots=\eta_n=0$ and $H(G)=m$. Thus the equalities in~(\ref{e1}) holds.
	Conversely, suppose the right equality holds. Then  from equation~(\ref{e2}), $\eta_1^2=\eta_2^2=\dots=\eta_n^2$. This implies that $G$ has at most two distinct $\mathcal{ABS}$ eigenvalue. Therefore, by Theorems \ref{one eig} and \ref{complete 1}, either $G\cong K_n$ $(n>2)$ or $n$ is even and $G\cong \dfrac{n}{2}K_2$. If $G\cong K_n$, then $\eta_1^2\neq\eta_2^2$, a contradiction. Thus, $G\cong \dfrac{n}{2}K_2$ for an even integer $n$. Similarly, the left equality holds if and only if $G\cong \dfrac{n}{2}K_2$ for an even integer $n$.
\end{proof}
\begin{theorem}\label{newthm}
    Let $G$ be a graph of order $n$ with minimum degree $\delta$ and maximum degree $\Delta$. Then $\sqrt{\delta(\delta-1)}\le \eta_{1}(G)\le \sqrt{\Delta(\Delta-1)}$. Further, equality holds if and only if $G$ is a regular graph.
\end{theorem}
\begin{proof}
From \cite[Theorem 8.1.22]{horn2012matrix}, $\displaystyle\min_{1\le i\le n}\{R_i\}\le \eta_{1}(G)\le\max_{1\le i\le n}\left\{R_i\right\}$, where $R_i$ is the row sum of the $i$th row of $\mathcal{AS}(G)$. Moreover, the equality on both sides holds if and only if all the row sums of $\mathcal{AS}(G)$ are equal. Now, 
\begin{eqnarray*}\label{eeeee1}
\max_{1\le i\le n}\left\{R_i\right\}=\max_{1\le i\le n } \sum\limits_{v_i:v_i\sim v_j}\sqrt{1-\frac{2}{d_i+d_j}}&\le \sqrt{\Delta(\Delta-1)}, 
\end{eqnarray*} 
where the equality holds if and only if one of the components of $G$ is a $\Delta$-regular graph,
and
\begin{eqnarray*}\label{eeeee2}
\min_{1\le i\le n}\left\{R_i\right\}=	\min_{1\le i\le n } \sum\limits_{v_{i}:v_i\sim v_j}\sqrt{1-\frac{2}{d_i+d_j}}&\ge \sqrt{\delta(\delta-1)},
\end{eqnarray*} 
where the equality holds if and only if one of the components of $G$ is a $\delta$-regular graph. Now, by Lemma \ref{eeeelemma}, the row sums of $\mathcal{AS}(G)$ is a constant if and only if $G$ is regular. Thus, 
$\sqrt{\delta(\delta-1)}\le \eta_{1}(G)\le\sqrt{\Delta(\Delta-1)}$ and the equality on both sides holds if and only if $G$ is a regular graph.
\end{proof}
The sum connectivity matrix of a graph $G$ is a general extended adjacency matrix with $\mathcal{F}(d_{i},d_{j})=\dfrac{1}{\sqrt{d_{i}+d_{j}}}$. It is denoted by $\mathcal{S}(G)$. In the following theorem, we give a relation between the spectral radius of $\mathcal{S}(G)$ ($\rho(\mathcal{S}(G))$) and $\eta_{1}(G)$.
\begin{theorem}\label{sumrand 1}
	If $G$ is a connected graph of order $n\ge2$ , then \begin{center}
		$\displaystyle\rho(\mathcal{S}(G))~\min_{\substack{\\v_i\sim v_j}}\sqrt{d_i+d_j-2}\le \eta_1(G)\le \rho(\mathcal{S}(G))\max_{\substack{\\v_i\sim v_j}}\sqrt{d_i+d_j-2}$. 
	\end{center} Further, equality on both sides holds if and only if $G$ is a regular graph or semiregular graph.
\end{theorem}
\begin{proof}
	The matrices $\mathcal{AS}(G)$ and $\mathcal{S}(G)$ are non-negative and irreducible. Moreover, \begin{equation*}
		\mathcal{S}(G)\min_{\\\substack{v_i\sim v_j}}\sqrt{d_i+d_j-2}\preceq \mathcal{AS}(G)\preceq \mathcal{S}(G)\max_{\\\substack{v_i\sim v_j}}\sqrt{d_i+d_j-2}.
	\end{equation*}
	Thus, by Lemma \ref{b_1b_2},\begin{equation*}
		\rho(\mathcal{S}(G))\min_{\\\substack{v_i\sim v_j}}\sqrt{d_i+d_j-2}\le \eta_{1}(G)\le \rho(\mathcal{S}(G))\max_{\\\substack{v_i\sim v_j}}\sqrt{d_i+d_j-2}.
	\end{equation*}
Now we consider the equality case. Suppose $\displaystyle\eta_1(G)=\rho(\mathcal{S}(G))\max_{v_i\sim v_j}\sqrt{d_i+d_j-2}$. Then by Lemma \ref{b_1b_2}, $\displaystyle\mathcal{AS}(G)=\mathcal{S}(G)\max_{v_i\sim v_j}\sqrt{d_i+d_j-2}$. This implies that, for every edge $v_iv_j$ in $G$, $\displaystyle\sqrt{d_i+d_j-2}=\max_{v_i\sim v_j}\sqrt{d_i+d_j-2}$. Therefore, $d_{i}+d_{j}$ is constant for any edge $v_{i}v_{j}$ of $G$.  Let $u$ (resp. $v$) be a vertex in $G$ with maximum degree $\Delta$  (resp. minimum degree $\delta$). Let $u\sim u_{1}$ and $v\sim v_{1}$. Then $d(u)+d(u_{1})\ge \Delta+\delta \ge d(v)+d(v_{1})$. Hence, $d_i+d_j=\delta+\Delta$, for all $v_iv_j\in E(G)$. Now, suppose there exists a vertex $w$ in $G$ such that $d(w)\in \left(\delta,~ \Delta\right)$. Then $G$ has a component whose vertex degrees are either $d(w)$ or $\Delta+\delta-d(w)$. Therefore, $G$ is disconnected,  a contradiction.  Thus, $G$ is $\Delta$-regular or $(\delta,\, \Delta)$-semiregular graph. Similarly, if $\displaystyle\eta_1(G)=\rho(\mathcal{S}(G))\min_{v_i\sim v_j}\sqrt{d_i+d_j-2}$, then $G$ is $\Delta$-regular or $(\delta,\, \Delta)$-semiregular graph. The converse part is straightforward.
\end{proof}

\begin{theorem}
    If $G$ is a connected graph with maximum degree $\Delta$ and minimum degree $ \delta$, then $\dfrac{2\chi(G)}{n}\sqrt{2\delta-2}\le \eta_{1}(G)\le\sqrt{\dfrac{2(\Delta-1)(n-1)}{n}~R(G)}$, where the equality on the left side holds only if $G$ is regular and the equality on the right side holds only if $G$ is a complete graph.
\end{theorem}
\begin{proof}
From \cite[Corollary 1]{zhou2010sum}, we have
\begin{equation}\label{neweq2}\dfrac{2\chi(G)}{n}\le \rho(\mathcal{S}(G))\le \sqrt{\dfrac{n-1}{n}R(G)},\end{equation}
where the left side equality holds only if  $\mathcal{S}(G)$ has equal row sums, and the right equality holds only if $G$ is a complete graph.   Therefore by Theorem \ref{newthm}, we get the desired result.	
\end{proof}
Now, we obtain an upper bound for the spectral radius of unicyclic graphs using the following lemma.
\begin{lemma}\cite{brouwer2011spectra}\label{PBT}
	Let $P\ge 0$ be an irreducible matrix with an eigenvalue $\theta$.If $PY\le tY$ for $t\in \mathbb{R}, Y\in \mathbb{R}^n$ and $Y\ge 0$, then $t\ge \theta$.
\end{lemma}
\begin{theorem}
	Let $G$ be a unicyclic graph of order $n$ and girth at least 5. Then $\eta_1(G)\le \sqrt{\dfrac{(n-3)(n+1)}{n-1}}$.
\end{theorem}
\begin{proof}
	Since $G$ is unicyclic of order $n$,  $\sum\limits_{i=1}^n d_i=2n$.\\ Therefore, \begin{align*}\sum\limits_{v_j:v_i\sim v_j} d_j&=2n-d_i-\sum\limits_{v_j:v_j\not\sim v_i}  d_j\le 2n-d_i-\sum\limits_{v_j:v_j\not\sim v_i} 1 =n+1.\end{align*} 
	That is, $\sum\limits_{v_j:v_i\sim v_j} d_j\le n+1$. Also, $\displaystyle d_{i}+d_{j}=2n-\sum_{\underset{k\neq i, j}{k=1}}^{n}d_{k}$. Since $G$ is of girth at least 5, there exist at least 3 vertices $v_p,v_q$ and $v_r$ on the cycle which are distinct from the vertices $v_{i}$ and $v_{j}$. Therefore, $d_i+d_j\le 2n-6-\sum\limits_{k\notin \left\{i,j,p,q,r\right\}}d_k\le n-1$. That is, $d_i+d_j\le n-1$. 
	Let $\Delta_a=max\left\{d_i+d_j: v_iv_j\in E(G)\right\}$,  $Y=\left\{\sqrt{d_1},\sqrt{d_2},\dots,\sqrt{d_n}\right\}$ and $\left(\mathcal{AS}Y\right)_i$ denote the $i^{th}$ row sum of the matrix $(\mathcal{AS})Y$. 
	Then $\Delta_{a}\le n-1$, and so 
	\begin{center}
		$\left(\mathcal{AS}Y\right)_i=\sum\limits_{v_{j}:v_{i}\sim v_{j}}\sqrt{1-\dfrac{2}{d_i+d_j}}\sqrt{d_j}\le \sqrt{1-\dfrac{2}{n-1}}\sum\limits_{v_{j}:v_{i}\sim v_{j}}\sqrt{d_j}.$\end{center}
	Now by Cauchy-Schwarz inequality, \begin{center}
		$\sum\limits_{v_{j}:v_{i}\sim v_{j}}\sqrt{d_j}\le \sqrt{\sum\limits_{v_{j}:v_{i}\sim v_{j} }d_j}\sqrt{d_i}\le \sqrt{n+1}\sqrt{d_i}$.\end{center}
	Therefore, $\left(\mathcal{AS}Y\right)_i\le \sqrt{\dfrac{(n-3)(n+1)}{n-1}}\sqrt{d_i}$, and thus
	$\mathcal{AS}Y\le \sqrt{\dfrac{(n-3)(n+1)}{n-1}}~ Y.$
	Hence, by Lemma~\ref{PBT}, $\eta_1(G)\le \sqrt{\dfrac{(n-3)(n+1)}{n-1}}$.
\end{proof}
\section{ Properties of Atom-bond sum-connectivity energy}\label{sec5}
In this section, we present some bounds on $\mathcal{E_{ABS}}(G)$.
\begin{theorem}
	Let $G$ be a graph with $n \ge 2$ vertices and $m$ edges. Then \begin{enumerate}[(i)]
	\item
	 $\mathcal{E}_{\mathcal{ABS}}(G)\ge2\sqrt{m-H(G)}$. Equality holds if and only if  $G\cong pK_{n_{1},n_{2}}\cup qK_{2}\cup rK_{1}$, where $n_{1}+n_{2}>2$,  $p=0$ or $1$ and $p(n_{1}+n_{2})+2q+r=n$.
	 \item
	 $\mathcal{E}_{\mathcal{ABS}}(G)\le\sqrt{2n(m-H(G))}$. Equality holds if and only if $G\cong pK_{2}\cup qk_{1}$, where $2p+q=n$.
	 \end{enumerate}
\end{theorem}
\begin{proof} (i) From \cite[Theorem 4]{das2018degree}, we have $\mathcal{E}_{\mathcal{ABS}}(G)\ge \sqrt{2trace(\mathcal{AS}^{2}(G))}$ and the equality holds if and only if $\eta_{1}=-\eta_{n}$ and $\eta_{2}=\eta_{3}=\cdots=\eta_{n-1}=0$. Since $trace(\mathcal{AS}^{2}(G))=2m-H(G)$, we get $\mathcal{E}_{\mathcal{ABS}}(G)\ge2\sqrt{m-H(G)}$.  Suppose $\eta_{1}=-\eta_{n}$ and $\eta_{2}=\eta_{3}=\cdots=\eta_{n-1}=0$. Then $G$ is bipartite  (see, Corollary \ref{bipartite}). Also, if $H$ is a component of $G$, then $\mathcal{AS}(H)$ has either two or three  distinct eigenvalues, or all its eigenvalues are equal to 0. Thus, from Theorems \ref{cbipartite}, we get $H\cong K_{n_{1},n_{2}}$ with $n_{1}+n_{2}>2$, $K_{2}$ or $K_{1}$. Furthermore, if $K_{n_{1},n_{2}}~(n_{1}+n_{2}>2)$ is a component of $G$, then all other components of $G$ are either $K_{2}$ or $K_{1}$. Otherwise $\eta_{2}>0$, a contradiction.\\[2mm]
(ii) From\cite[Corollary 2]{das2018degree}, we have $\mathcal{E}_{\mathcal{ABS}}(G)\le \sqrt{n\, trace(\mathcal{AS}^{2}(G))}$ and the equality holds if and only of $|\eta_{1}|=|\eta_{2}|=\cdots=|\eta_{n}|$. Therefore, $\mathcal{E}_{\mathcal{ABS}}(G)\le \sqrt{2n\,(m-H(G))}$. Suppose $|\eta_{1}|=|\eta_{2}|=\cdots=|\eta_{n}|$. Then the eigenvalues of $\mathcal{AS}(G)$ are all equal or it has exactly two distinct eigenvalues. So, by Theorems \ref{one eig} and \ref{complete 1}, each component of $G$ is $K_{n_{1}}$ for some positive integer $n_{1}$. Furthermore, $n_{1}=1$ or $2$. Otherwise $\eta_{1}>\eta_{2}$. This completes the proof. 
\end{proof}
To prove our next upper bound on $\mathcal{E}_{\mathcal{ABS}}(G)$, we need the following lemma.
\begin{lemma}\label{strongly}\cite{cvetkovic1980spectra}
A regular connected graph $G$ is strongly regular if and only if it has three distinct eigenvalues.
\end{lemma}
\begin{theorem}
	Let G be a graph of order n with m edges. 
    If $m=H(G)$ or $2(m-H(G))\ge n$, then $\mathcal{E}_{\mathcal{ABS}}(G)\le \dfrac{2\mathcal{ABS}(G)}{n}+\sqrt{(n-1)\left(2m-2H(G)-\dfrac{4 (\mathcal{ABS}(G))^2}{n^2}\right)}$. Further, equality holds if and only if $G\cong pK_{2}\cup qK_{1}$, where $2p+q=n$, $G\cong K_{n}$ or $G$ is a non-complete strongly regular graph.
\end{theorem}
\begin{proof}
     Using Cauchy-Schwarz inequality, \begin{align*}
        \mathcal{E}_{\mathcal{ABS}}(G)=\sum\limits_{i=1}^n |\eta_i|&=\eta_1+\sum\limits_{i=2}^n |\eta_i|
        \le \eta_1+\sqrt{(n-1)(2m-2H(G)-\eta_1^2)},  
    \end{align*} where the equality holds if and only if $|\eta_{2}|=|\eta_{3}|=\cdots=|\eta_{n}|$. Let $\displaystyle g(x)=x+\sqrt{(n-1)(2m-2H(G)-x^2)}$. Then by first derivative test, the function $g$ is decreasing for $\sqrt{\dfrac{2(m-H(G))}{n}}\le x\le \sqrt{2(m-H(G))}.$  
Now, $\dfrac{2\mathcal{ABS}(G)}{n}=\dfrac{2}{n}\sum\limits_{v_iv_j\in E(G)}\sqrt{\dfrac{d_i+d_j-2}{d_i+d_j}}\ge \dfrac{2}{n}\sum\limits_{v_iv_j\in E(G)}\dfrac{d_i+d_j-2}{d_i+d_j}=\dfrac{2(m-H(G))}{n}\ge \sqrt{\dfrac{2(m-H(G))}{n}}$ (because $2(m-H(G))\ge n$). That is, $\dfrac{2\mathcal{ABS}(G)}{n}\ge \sqrt{\dfrac{2(m-H(G))}{n}}$.  Upon combining the above inequality with Theorem \ref{lbe2}, we get
\begin{equation}
	\sqrt{\dfrac{2(m-H(G))}{n}}\le \dfrac{2 \mathcal{ABS}(G)}{n}\le \eta_{1}\le \sqrt{2(m-H(G))}.
\end{equation}
 Therefore, \begin{equation}\label{eta}
 \mathcal{E}_{\mathcal{ABS}}(G)\le g(\eta_1)\le g\left(\dfrac{2\mathcal{ABS}(G)}{n}\right)= \dfrac{2\mathcal{ABS}(G)}{n}+\sqrt{(n-1)\left(2m-2H(G)-\dfrac{4 (\mathcal{ABS}(G))^2}{n^2}\right)}
. \end{equation}
  Suppose the equality in equation (\ref{eta}) holds. Then $\eta_1=\dfrac{2\mathcal{ABS}(G)}{n}$ and $|\eta_2|=|\eta_{3}|=\cdots=|
    \eta_{n}|$. Thus, $\mathcal{AS}(G)$ has at most three distinct eigenvalues. If  $\mathcal{AS}(G)$ has at most two distinct eigenvalues, then by Theorems \ref{one eig} and \ref{complete 1}, $G\cong pK_{2}\cup qK_{1}$, where $p+q=n$, or $G\cong K_{n}$. Otherwise, $\mathcal{AS}(G)$ has exactly three distinct eigenvalues. Now, by Theorem~\ref{lbe2}, $G$ is  $k$-regular graph for some constant $k$, and so $\mathcal{AS}(G)=\sqrt{\dfrac{k-1}{k}}A(G)$. If $\mathcal{AS}(G)$ has exactly three distinct eigenvalues, then $G$ has exactly three distinct  eigenvalues. Therefore by Lemma \ref{strongly}, $G$ must be a non-complete strongly regular graph. Conversely, if $G\cong pK_{2}\cup qK_{1}$, where $2p+q=n$, or $G\cong K_{n}$, then one can easily see that the equality in (\ref{eta}) holds. Suppose $G$ is a non-complete strongly $k$-regular graph, then $\eta_{1}=\sqrt{k(k-1)}$, $|\eta_{j}|=\sqrt{\dfrac{(n-k)(k-1)}{n-1}}$ for $j=2,3,\ldots,n$. Now, one can easily check that equality in (\ref{eta}) holds. This completes the proof of the theorem. 
    \end{proof} 
The following lemmas are useful to prove our next result.

\begin{lemma}{\rm\cite{merikoski2003characterizations}}\label{spread}
	If $M$ is a Hermitian $n\times n$ matrix, then $\displaystyle|\theta_1(M)-\theta_n(M)|\ge 2\max_{j}\left(\sum\limits_{k:k\neq j}|a_{jk}|^2\right)^{\frac{1}{2}}$.
\end{lemma}

\begin{lemma}\label{spread bound}
	Let $G$ be a connected graph of order $n\ge 2$ with maximum degree $\Delta$ and minimum degree $\delta$. Then
	\begin{center}
		$2\sqrt{\dfrac{\Delta\left(\Delta+\delta-2\right)}{\Delta+\delta}}\le \eta_1+|\eta_n|\le  2\sqrt{m-H(G)}$.\end{center}
\end{lemma}
\begin{proof}
	By Lemma~\ref{spread}, \begin{equation*}\displaystyle \eta_1+|\eta_n|\ge 2~ \max_i\left(\sum\limits_{v_{j}:v_i\sim v_j}\frac{d_i+d_j-2}{d_i+d_j}\right)^{1/2}\ge 2\sqrt{\frac{\Delta(\Delta+\delta-2)}{\Delta+\delta}}.\end{equation*}
	Proving the left inequality.
	Now, by Cauchy-Schwarz inequality and from Proposition \ref{ABS prop1}, \begin{equation*}
		\eta_1+|\eta_n|\le \sqrt{2(\eta_1^2+\eta_n^2)}\le \sqrt{4(m-H(G))}=2\sqrt{m-H(G)}.  
	\end{equation*}
\end{proof}
\begin{theorem}
Let $G$ be  $(n, m)$-graph with maximum degree $\Delta$.  	
If $\dfrac{2\left(m-H(G)\right)}{n}\le \dfrac{\Delta\left(\Delta-1\right)}{\Delta+1}$, $\mathcal{E}_{\mathcal{ABS}}(G)\le 2\sqrt{\dfrac{\Delta\left(\Delta-1\right)}{\Delta+1}}+\sqrt{2(n-2)\left(m-H(G)-\dfrac{\Delta\left(\Delta-1\right)}{\Delta+1}\right)}$. Equality holds if and only if $G\cong pK_{1,\Delta}\cup qK_{2}\cup rK_{1}$, where $p=0$ or $1$ and $p(\Delta+1)+2q+r=n$.
\end{theorem}
\begin{proof}
Let $\eta_{1}\ge \eta_{2}\ge\cdots\ge \eta_{n}$ be the $\mathcal{ABS}$ eigenvalues of $G$. Using Cauchy-Schwarz inequality, 
    \begin{align*}
        \mathcal{E}_{\mathcal{ABS}}(G)=\eta_1+|\eta_n|+\sum\limits_{i=2}^{n-1}|\eta_i|\le \eta_1+|\eta_n|+\sqrt{\sum\limits_{i=2}^{n-1}(n-2)|\eta_i|^{2}},
       \end{align*}
   where the equality holds if and only if $|\eta_{2}|=|\eta_{3}|=\cdots=|\eta_{n-1}|$.
   Therefore, by Proposition \ref{ABS prop1},
   \begin{align*} 
    \mathcal{E}_{\mathcal{ABS}}(G) &\le \eta_1+|\eta_n|+\sqrt{(n-2)(2m-2H(G)-\eta_1^2-\eta_n^2)}\label{cs2}.
  \end{align*}
Further, by A.M.-G.M. inequality, $2\sqrt{\eta_{1}\eta_{n}}\le\eta_{1}+|\eta_{n}|$ and the equality holds if and only if $\eta_{1}=|\eta_{n}|$, Thus
\begin{align*}    
   \mathcal{E}_{\mathcal{ABS}}(G) &\le \eta_1+|\eta_n|+\sqrt{(n-2)\left(2m-2H(G)-\frac{\left(\eta_1+|\eta_n|\right)^2}{2}\right)}.
    \end{align*}
 Let $f(x)=2x+\sqrt{(n-2)(2m-2H(G)-2x^2)}$. Then $f$ is decreasing for $\sqrt{\dfrac{2(m-H(G))}{n}}\le x\le \sqrt{m-H(G)}$. By Lemma~\ref{spread bound}, \begin{equation*}
        \sqrt{\frac{2(m-H(G))}{n}}\le \sqrt{\frac{\Delta(\Delta-1)}{\Delta+1}}\le \frac{\eta_1+|\eta_n|}{2}\le  \sqrt{m-H(G)}.
    \end{equation*} So,\begin{align}\label{cs7}
        E_{\mathcal{ABS}}(G)&\le f\left(\frac{\eta_1+|\eta_n|}{2}\right)\le f\left(\sqrt{\frac{\Delta(\Delta-1)}{\Delta+1}}\right)\notag\\
        &=2\sqrt{\frac{\Delta(\Delta-1)}{\Delta+1}}+\sqrt{(n-2)\left(2m-2H(G)-2\frac{\Delta(\Delta-1)}{\Delta+1}\right)}.
    \end{align} 
Suppose the equality in equation (\ref{cs7}) holds. Then $\eta_{1}=|\eta_{n}|=\sqrt{\dfrac{\Delta(\Delta-1)}{\Delta+1}}$ and $|\eta_{2}|=|\eta_{3}|=\cdots=|\eta_{n-1}|$. Thus, by Perron-Frobenius theorem, $G$ is a bipartite graph. 
If $\Delta=1$, then, $G\cong qK_{2}\cup rK_{1}$, where $2q+r=n$. Otherwise, $\Delta>1$, and so $\eta_{1}=|\eta_{n}|>0$ . Let $H$ be component of $G$ having a vertex of degree $\Delta$. Since $G$ is bipartite, $K_{1,\Delta}$ is an induced subgraph of $H$. So, by Cauchy's interlacing theorem,
$\eta_{1}(H)\ge \eta_{1}(K_{1,\Delta})=\sqrt{\dfrac{\Delta(\Delta-1)}{\Delta+1}}$.  Moreover the equality holds if and only if $H\cong K_{1,\Delta}$.  Now, $\eta_{1}=\sqrt{\dfrac{\Delta(\Delta-1)}{\Delta+1}}\ge\eta_{1}(H)$ because $H$ is a component of $G$. Therefore, $\eta_{1}(H)=\eta_{1}=\dfrac{\Delta(\Delta-1)}{\Delta+1}$ and $H\cong K_{1,\Delta}$. Further,   $0$ is an $\mathcal{ABS}$ eigenvalue of $H$, and thus  $\mathcal{ABS}$ eigenvalues of $G$ are $\eta_{1}=|\eta_{n}|=\sqrt{\dfrac{\Delta(\Delta-1)}{\Delta+1}}$ and $|\eta_{2}|=|\eta_{3}|=\cdots=|\eta_{n-1}|=0$. Therefore, if $H_{1}\ncong H$ is a component of $G$, then all its $\mathcal{ABS}$ eigenvalues are equal to $0$. Therefore, by Theorem \ref{one eig}, $H_{1}\cong K_{2}$ or $K_{1}$. Thus $G\cong pK_{1,\Delta}\cup qK_{2}\cup rK_{1}$, where $p=0$ or $1$ and $p(\Delta+1)+2q+r=n$. Conversely, if  $G\cong pK_{1,\Delta}\cup qK_{2}\cup rK_{1}$, where $p=0$ or $1$ and $p(\Delta+1)+2q+r=n$, then one can easily verify that the equality holds.
\end{proof}
\begin{theorem}
    Let $\varphi_1\ge\varphi_2\ge\cdots\ge\varphi_n$ be the $\mathcal{ABC}$-eigenvalues and $\eta_1\ge\eta_2\ge\cdots\ge\eta_n$ be the $\mathcal{ABS}$ eigenvalues of a graph $G$ without pendent vertices. Then $\mathcal{E}_{\mathcal{ABS}}(G)\ge \sqrt{\dfrac{2}{n}}\mathcal{E}_{\mathcal{ABC}}(G)$.
\end{theorem}
\begin{proof}
    We have, 
   \allowdisplaybreaks
   \begin{align*}\label{15}
        (\mathcal{E}_{\mathcal{ABS}}(G))^2&=\left(\sum\limits_{i=1}^n |\eta_i|\right)^{2}=\sum\limits_{i=1}^n \eta_i^2+2\sum\limits_{i<j}|\eta_i||\eta_j| \notag \\
        &\ge \sum\limits_{i=1}^n \eta_i^2+2 \left|\sum\limits_{i<j}\eta_i\eta_j\right|~~ \text{(by triangle inequality)}\notag \\
        &= 2 \sum\limits_{i=1}^n \eta_i^2 ~~~\left(\text{because} \sum\limits_{i=1}^n \eta_i^2=-2 \sum\limits_{i<j}\eta_i\eta_j\right)\notag\\
        &=4\sum_{v_{i}v_{j}\in E(G)}\dfrac{d_{i}+d_{j}-2}{d_{i}+d_{j}}\notag \\
        &\ge 4\sum_{v_{i}v_{j}\in E(G)}\dfrac{d_{i}+d_{j}-2}{d_{i}d_{j}}\\
        &=2\sum_{i=1}\varphi_{i}^{2}\ge \dfrac{2}{n}\left(\sum_{i=1}|\varphi_{i}|\right)^{2} (\text{by Cauchy-Schwarz inequality})\\
        &=\dfrac{2}{n}\left(\mathcal{E_{ABC}}(G)\right)^{2}.
    \end{align*} 
Thus, $\mathcal{E}_{\mathcal{ABS}}(G)\ge \sqrt{
\dfrac{2}{n}}\mathcal{E_{ABC}}(G)$.
\end{proof}
 \section{QSPR analysis of benzenoid hydrocarbon}
In this section, we show that the physicochemical properties, namely, the boiling point (BP) and pi-electron energy ($\mathcal{E}_{\pi}$)of benzenoid hydrocarbons can be modeled using $\mathcal{ABS}$-energy.
 The experimental values listed in this section are taken from \cite{mondal2023degree,das2022ve,ramane2017status}. The hydrogen-suppressed molecular graphs are depicted in Figure~\ref{benzoid}. 
\begin{figure}[h]
	\centering
\includegraphics[width=0.9\textwidth]{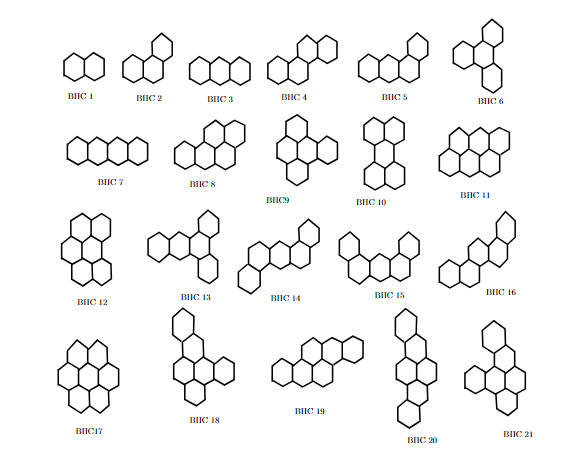}
	\caption{Hydrogen-suppressed molecular graph of benzenoid hydrocarbons}
	\label{benzoid}
\end{figure}
 The calculated values of $\mathcal{ABS}$ energy for benzenoid hydrocarbons are shown in Table~\ref{table} . 
\begin{table}[H]
	\begin{minipage}[c]{0.45\textwidth}
		\begin{tabular}{|c|c|c|c|}
			\hline 
			Compound & $\mathcal{E}_{\mathcal{ABS}}$ & BP  & $\mathcal{E}_{\pi}$ \\ \hline
			BHC1&	10.089&	218&	13.6832\\
			BHC2&	14.5445&	338&	19.4483\\
			BHC3&	14.4865&	340&	19.3137\\
			BHC4&	18.9799	&431	&25.1922\\
			BHC5&	18.9015	&425	&25.1012\\
			BHC6&	19.0111	&429	&25.2745\\
			BHC7&	18.8744	&440	&24.9308\\
			BHC8&	21.5036	&496	&28.222\\
			BHC9&	21.5786	&493	&28.3361\\
			BHC10&	21.4939	&497	&28.2453\\
			BHC11&	24.0287	&547	&31.253\\
			\hline
		\end{tabular}
	\end{minipage}
	\begin{minipage}[c]{0.45\textwidth}
		\begin{tabular}{|c|c|c|c|}
			\hline
			Compound & $\mathcal{E}_{\mathcal{ABS}}$ & BP  & $\mathcal{E}_{\pi}$ \\ \hline
			BHC12&	23.5972&	542&	31.4251\\
			BHC13&	23.4346	&535	&30.9418\\
			BHC14&	23.4184	&536	&30.8805\\
			BHC15&	23.3769	&531	&30.8795\\
			BHC16&	23.4463	&519	&30.9432\\
			BHC17&	26.7119	&590	&34.5718\\
			BHC18&	25.9835	&592	&34.0646\\
			BHC19&	25.915	&596	&33.1892\\
			BHC20&	25.9357	&594	&33.9542\\
			BHC21&	25.9565	&595	&34.0307\\
			-&-&-&-\\
			\hline
		\end{tabular}
	\end{minipage}
	\caption{Experimental physicochemical properties and theoretical $\mathcal{ABS}$ energy of benzenoid hydrocarbons.}
	\label{table}
\end{table}
Consider the following model:
\begin{equation}\label{qspr}
	Y= A (\pm S_e)\mathcal{E}_{\mathcal{ABS}}+B(\pm S_e),
\end{equation}
 where $Y, A, S_e$ and $B$ denote the property, slope, standard error of coefficients and intercept, respectively. We denote  the correlation coefficient,  standard error of the model, the $F$-test value  and the significance by $r$, $SE$, $F$ and $SF$, respectively.\\ For benzenoid hydrocarbons, it is found that the $\mathcal{ABS}$ energy has a strong correlation with the boiling point and pi-electron energy. In fact, we get the following regression equations for benzenoid hydrocarbons using model~(\ref{qspr}).
\begin{equation}\label{BP}
	BP=22.684(\pm 0.4013)\mathcal{E}_{\mathcal{ABS}}+2.25(\pm 8.7952),\end{equation}
~~~~~~~~~~~~~~~~~$r^2=0.9941,~ SE=7.8838, ~F=3194.002,~ SF=1.23\times 10^{-22}$.
\begin{figure}[H]
	\includegraphics[width=6.5in, height=2in]{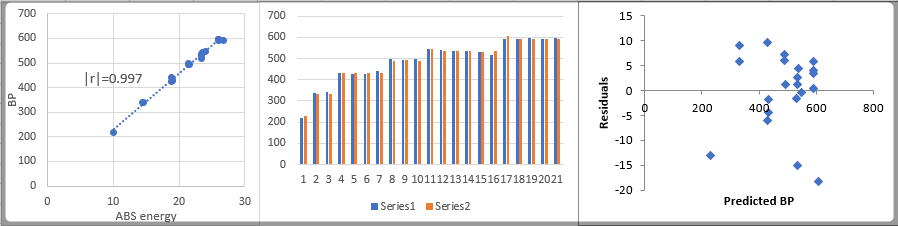}
	\caption{Linear relation of $\mathcal{ABS}$ energy with BP, experimental and predicted BP and residual plot.}
\end{figure}
\begin{equation}\label{Pi}
	\mathcal{E}_{\pi}=-1.2668(\pm 0.01235)\mathcal{E}_{\mathcal{ABS}}+1.0586(\pm 0.2706),
\end{equation}
~~~~~~~~~~~~~~~~$r^2=0.9982,~ SE=0.2425, F=10522.29,~ SF=1.54\times 10^{-27}$.
\begin{figure}[H]
	\includegraphics[width=6.5in, height=2in]{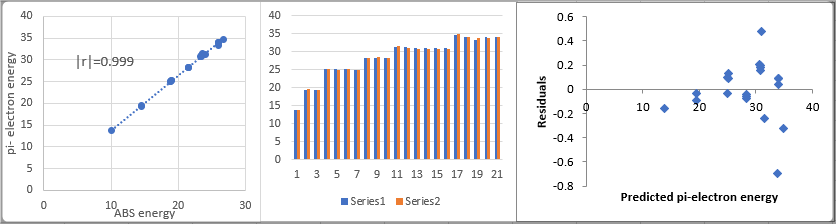}
	\caption{Linear relation of $\mathcal{ABS}$ energy  with pi-electron energy, experimental and predicted pi-electron energy and residual plot.}
\end{figure}
The data variance for BP and pi-electron energy is around 99\%. The standard errors are very low, particularly in model~(\ref{Pi}), where they are significantly small. This low standard error enhances the model's consistency and increases the F-value, especially for pi-electron energy. The SF values are significantly below 0.05.  The predicted properties from model~(\ref{qspr}) are compared with the experimental properties using bar diagrams, where series 1 is related to experimental value and series 2 is related to predicted value. These figures show that the experimental and predicted data align well. Additionally, the residuals are randomly scattered around the zero line, indicating that the model is consistent.\\
In \cite{mondal2023degree,das2022ve,ramane2017status,liu2021more,nithya2023smallest} the QSPR analysis of benzenoid hydrocarbons is done using the second-degree based entropy, ve-degree irregularity index, Albertson index, first and second status connectivity indices, first and second eccentric connectivity indices, Wiener index, Sombor index, reduced Sombor index and $\mathcal{ABS}$ index. It is observed that the $|r|$ value obtained for BP using $\mathcal{E}_{\mathcal{ABS}}$ is better than that of $|r|$ value obtained from these indices. Further, $\mathcal{E}_{\mathcal{ABS}}$ have high pi-electron energy predictive ability compared to second-degree based entropy and $\mathcal{ABS}$ index. With the smaller standard error and higher $F$-value of the proposed models, we can conclude that the performance of the models is better than that of the models discussed in~\cite{nithya2023smallest} using $\mathcal{ABS}$ index.
\section{Conclusions}
In this work, we have determined all connected graphs with $\eta_{n}>-1$. As a result, graphs with two distinct $\mathcal{ABS}$ eigenvalues are classified. Also, bipartite graphs with three distinct $\mathcal{ABS}$ eigenvalues are determined. Further, some bounds on the spectral radius and energy of the matrix $\mathcal{AS}(G)$ are obtained. The problem of characterizing non-bipartite graphs with three distinct $\mathcal{ABS}$ eigenvalues remains open. Also, the chemical importance of $\mathcal{ABS}$ energy is demonstrated. As a future work, the problem of obtaining sharp bounds for the spectral radius and energy of the matrix $\mathcal{AS}(G)$ in terms of graph parameters would be interesting.     


\begin{thebibliography}{99}
	\bibitem{akbariarelations} S. Akbari, M. Habibi, S. Rabizadeh, Relations between Energy and Sombor Index,
	\textit{MATCH Commun. Math. Comput. Chem.\/} \textbf{92}(2024) 425–435.
	\bibitem{alawiah2018new} N. Alawiah, N. J. Rad, A. Jahanbani, H. Kamarulhaili. New upper bounds on the energy of a graph, \textit{MATCH Commun. Math. Comput. Chem.\/} \textbf{79}(2018) 287–301.
	\bibitem{ali2020symmetric} A. Ali, S. Elumalai, T. Mansour, On the symmetric division deg index of molecular graphs, \textit{MATCH Commun. Math. Comput. Chem.\/}  \textbf{83} (2020) 205–220.
	\bibitem{ali2022atom} A. Ali, B. Furtula, I.Red{\v{z}}epovi{\'c}, I. Gutman, Atom-bond sum-connectivity
	index, \textit{Journal of Mathematical Chemistry}. \textbf{60} (2022) 2081–2093.
	\bibitem{ali2024extremal} A. Ali, I. Gutman, B. Furtula, I. Red{\v{z}}epovi{\'c}, T. Do{\v{s}}li{\'c}, Z. Raza, Extremal
	results and bounds for atom-bond sum-connectivity index, \textit{MATCH Commun. Math. Comput. Chem.\/} \textbf{92}(2024) 271–314.
	\bibitem{ali2023atom} A. Ali, I. Gutman, and I.Red{\v{z}}epovi{\'c}, Atom-bond sum-connectivity index of unicyclic
	graphs and some applications, \textit{Electron. J. Math}. \textbf{5}(2023) 1–7.
	\bibitem{brouwer2011spectra} A. E. Brouwer, W. H. Haemers, \textit{Spectra of graphs}, Springer, New York, 2011.
	\bibitem{cvetkovic1980spectra} D. M. Cvetkovi{\'c}, M. Doob, H. Sachs, \textit{Spectra of graphs: theory and applications}, Academic Press, New York, 1980.
	\bibitem{tavakoli2024energy} K. C. Das, G. Ali, M. Tavakoli, On the Energy and Spread of the Adjacency, Laplacian and Signless Laplacian Matrices of Graphs, \textit{MATCH Commun. Math. Comput. Chem.\/} \textbf{92}(2024) 545–566.
	\bibitem{das2018degree} K. C. Das, I. Gutman, I. Milovanovi{\'c}, E. Milovanovi{\'c}, B. Furtula, Degree-based
	energies of graphs,\textit{ Linear algebra and its applications}. \textbf{554}(2018) 185–204.
	\bibitem{das2022ve} K. C. Das, S. Mondal, On ve-degree irregularity index of graphs and its applications as molecular descriptor, \textit{Symmetry}. \textbf{14}(2022) 2406.
	\bibitem{das2023neighborhood} K. C. Das, S. Mondal, On neighborhood inverse sum indeg index of molecular
	graphs with chemical significance,\textit{ Information Sciences}. \textbf{623}(2023) 112–131.
	\bibitem{espinalgraph} C. Espinal, J. Rada, Graph energy change due to vertex deletion, \textit{MATCH Commun. Math. Comput. Chem.\/} \textbf{92}(2024) 89–103.
	\bibitem{estrada2017meaning}E. Estrada, M. Benzi, What is the meaning of the graph energy after all?, \textit{Discrete
		Applied Mathematics}. \textbf{230}(2017) 71–77.
	\bibitem{gutman2013degree}I. Gutman, Degree-based topological indices, \textit{Croatica chemica acta}. \textbf{86}(2013) 351–361.
	\bibitem{gutman2021geometric} I. Gutman, Geometric approach to degree-based topological indices: Sombor indices,
	\textit{MATCH Commun. Math. Comput. Chem.\/} \textbf{86}(2021) 11–16.
	\bibitem{gutman2020research} I. Gutman, H. Ramane, Research on graph energies in 2019,\textit{MATCH Commun. Math. Comput. Chem.\/}  \textbf{84}(2020) 277–292.
	\bibitem{horn2012matrix}R. A. Horn, C. R. Johnson, Matrix analysis, Cambridge university press, 2012.
	\bibitem{li2012graph} X. Li, Y. Shi, I. Gutman, Graph energy, Springer Science \& Business Media, 2012.
	\bibitem{lin2024abs} Z. Lin, T. Zhou,Y. Liu, On ABS Estrada index of trees, \textit{ Journal of Applied
		Mathematics and Computing}.(2024) 1–13.
	\bibitem{lin2024abstrees} Z. Lin, T. Zhou,Y. Liu, On the atom-bond sum-connectivity spectral radius of trees, \textit{Discrete Mathematics Letters}.(2024) 122–127.
	\bibitem{liu2021more}H. Liu, H. Chen, Q. Xiao, X. Fang, Z. Tang, More on Sombor indices of chemical graphs and their applications to the boiling point of benzenoid hydrocarbons,
	\textit{International Journal of Quantum Chemistry}. \textbf{121}(2021) 26689.
	\bibitem{merikoski2003characterizations}J. K. Merikoski, R. Kumar, Characterizations and lower bounds for the spread
	of a normal matrix, \textit{Linear Algebra and its Applications}. \textbf{364}(2003) 13–31.
	\bibitem{mondal2023degree} S. Mondal, K. C. Das, Degree-Based Graph Entropy in Structure--Property Modeling, \textit{Entropy}. \textbf{25}(2023) 1092.
	\bibitem{nithya2023smallest} P. Nithya, S. Elumalai, S. Balachandran, S. Mondal, Smallest abs index of
	unicyclic graphs with given girth, \textit{Journal of Applied Mathematics and Computing}. \textbf{69}(2023) 3675–3692.
	\bibitem{oboudi2019new} M. R. Oboudi, A new lower bound for the energy of graphs, \textit{Linear Algebra and its
		Applications}. \textbf{580}(2019) 384–395.
	\bibitem{ramane2017status} H. S. Ramane, A. S. Yalnaik, Status connectivity indices of graphs and its applications to the boiling point of benzenoid hydrocarbons, \textit{Journal of Applied
		Mathematics and Computing}. \textbf{55} (2017) 609–627.
	\bibitem{zhou2010sum} B. Zhou, N. Trinajstic, On sum-connectivity matrix and sum-connectivity energy of (molecular) graphs, \textit{Acta Chim. Slov}. \textbf{57} (2010) 518–523.
\end{thebibliography}
 \end{document}